\documentclass[a4paper]{amsart}

\usepackage{amsfonts}
\usepackage{amsmath}
\usepackage{amssymb}
\usepackage{amsthm,color}
\usepackage{enumerate}
\usepackage{mathtools}
\mathtoolsset{showonlyrefs}
 \usepackage{transparent}
 \usepackage{stackengine,tikz}

\usepackage{geometry}
\usepackage{color,graphicx,enumerate,wrapfig,amssymb,mathtools}
\usepackage{epsfig}
\usepackage{pxfonts}
\usepackage{graphicx}

\usepackage{eucal}

\newcommand{\R}{{\mathbb R}}

\newcommand{\cA}{{\mathcal A}}

\newcommand{\cH}{{\mathcal H}}

\newcommand{\cP}{{\mathcal P}}

\newcommand{\e}{\varepsilon}
\newcommand{\al}{\alpha}
\newcommand{\be}{\beta}

\newcommand{\la}{\lambda}

\newcommand{\loc}{\operatorname{loc}}

\newcommand{\D}{\nabla}

\DeclareRobustCommand{\rchi}{{\mathpalette\irchi\relax}}
\newcommand{\irchi}[2]{\raisebox{\depth}{$#1\chi$}}

\def\Xint#1{\mathchoice
  {\XXint\displaystyle\textstyle{#1}}%
  {\XXint\textstyle\scriptstyle{#1}}%
  {\XXint\scriptstyle\scriptscriptstyle{#1}}%
  {\XXint\scriptscriptstyle\scriptscriptstyle{#1}}%
  \!\int}
\def\XXint#1#2#3{{\setbox0=\hbox{$#1{#2#3}{\int}$}
    \vcenter{\hbox{$#2#3$}}\kern-.5\wd0}}

\def\dashint{\Xint-}

\newcommand{\qq}[1]{``#1''}

\newcommand{\mean}[1]{\langle{#1}\rangle}

\usepackage{amsmath}
\usepackage{amsfonts}
\usepackage{caption}
\usepackage{subcaption}
\usepackage{esint}

\newtheorem{theorem}{Theorem}
\theoremstyle{plain}

\newtheorem{lemma}{Lemma}
\newtheorem{remark}{Remark}
\newtheorem{proposition}{Proposition}
\newtheorem{conjecture}{Conjecture}
\numberwithin{equation}{section}

\makeatletter
\@namedef{subjclassname@2020}{%
  \textup{2020} Mathematics Subject Classification}
\makeatother

\begin{document}

\author{Seongmin Jeon}
\address[Seongmin Jeon]
{Department of Mathematics Education \newline 
\indent Hanyang University \newline 
\indent 222 Wangsimni-ro, Seongdong-gu, Seoul 04763, Republic of Korea} 
\email[Seongmin Jeon]{seongminjeon@hanyang.ac.kr}

\author{Henrik Shahgholian}
\address[Henrik Shahgholian]
{Department of Mathematics \newline 
\indent KTH Royal Institute of Technology \newline 
\indent 100 44 Stockholm, Sweden} 
\email[Henrik Shahgholian]{henriksh@kth.se}

\title[Fine Structure of Free boundaries]{Remarks on the fine structure of the free boundary \\ (the no-sign  obstacle problem) }

\date{\today}
\thanks{SJ was supported by the Academy of Finland grant 347550 and and the research fund of Hanyang University(HY-202400000003278). HS is supported by Swedish Research Council.}

\begin{abstract}

We present a number of results inspired by the approach developed in a recent work by A. Figalli and J. Serra on the fine structure of the obstacle problem, which turns out to be partially effective in addressing the no-sign obstacle problem. However, we also highlight one or two central questions that remain open and appear to require different techniques not presently available to us.

\end{abstract}

\maketitle 

\tableofcontents

\section{Introduction}

\subsection{Background}

The structure of free boundaries (both smooth and singular points) have been subject of intense study in the last few decades, and have given rise to several tools that have been developed as well as several results on various free boundary models. The study of the regular part of the free boundary, originating in the work of L.A. Caffarelli \cite{Caff-Acta}, has found many directions of flourishments, and the approach has been diversified to several directions. However, the study of singular set of the free boundary, again originated in the work of L.A. Caffarelli \cite{Caf98}--The more accurate citation would be works of Caffarelli-Rivi\`{e}re \cite{CR-1, CR-2} in two space dimensions where methods developed there were specific for the plane. 
M. Sakai gave a complete description of the free boundary, regular as well as singular parts, 
in the plane \cite{Sakai1, Sakai2}.
We also refer to work of G. Weiss \cite{Weiss99} and R. Monneau \cite{Monneau} for partial results.


A. Figalli and J. Serra \cite{FigSer19} undertook a detailed investigation of this problem, demonstrating that in the case $n = 2$, singular points are locally situated on a $C^2$ curve. For higher dimensions $n \geq 3$, they established that an analogous result holds, where singular points lie within $C^{1,1}$ manifolds (or within a countable collection of $C^2$ manifolds), except for certain "anomalous" points of higher codimension. Furthermore, they proved that the higher-dimensional stratum is necessarily contained in a $C^{1,\alpha}$ manifold, thereby generalizing to all dimensions a result previously obtained by G. Weiss \cite{Weiss99}. The result in \cite{FigSer19} was generalized to the fully nonlinear obstacle problem in \cite{SavYu23} and \cite{SavYu22}.

Our interest stemmed from investigating whether the Figalli-Serra approach could be applied to the no-sign obstacle problem \eqref{eq:no-sign-obst} and identifying the challenges that might arise.
Notably, Caffarelli’s approach \cite{Caf98}, with some modifications, was readily adaptable to the no-sign obstacle problem \cite{CafSha04}.

It soon became clear to us that almost every step of Figalli and Serra’s method translated effortlessly to our setting—except for one "minor" complication!
This difficulty, which we elaborate on later in this note, prevents us from establishing the fine structure of the top stratum $\Sigma_{n-1}$, as achieved in the classical obstacle problem \cite{FigSer19}; see Remark~\ref{rem:n-1} and compare Theorem~\ref{thm:main} (a) with Conjecture~\ref{conj:n-1}. This challenge is the primary motivation for writing this paper. In our unsuccessful attempt to resolve this issue, we encountered  other intriguing questions, which we will discuss in due course.

\subsection{The no-sign obstacle problem}
Consider the no-sign obstacle problem \begin{align}
    \label{eq:no-sign-obst}
    \begin{split}
        &\Delta u=\rchi_{\Omega(u)}\quad\text{in }B_1,\\
        &\Omega(u):=B_1\setminus\Lambda(u),
    \end{split}
\end{align}
where $\Lambda(u)=\{u=|\D u|=0\}$ and  $B_1$ is the unit ball in $n$-dimensional Euclidean space $\R^n$, ($n\geq 2$). 
This problem appears in many applications and has a non-variational nature, and its existence cannot be achieved by minimizing a (corresponding)  functional. 
This problem arises as an inverse source problem in potential theory and geophysics  \cite[Notes of Ch. 1]{PetShaUra12},  in scattering theory  \cite{SaSh}, and in many other problems. 
For results on the existence of solutions, we refer the reader to the book by V. Isakov \cite{Isakov}. In a more general setting, such as the so-called superconductivity problem—where the equation is modified to $\Delta u = \chi_{\{|\nabla u| > 0\}}$—we refer to the work of  L.A. Caffarelli and J.  Salazar \cite{CaSa}.

The aim of this paper is to study the fine structure of the free boundary for the no-sign obstacle problem \eqref{eq:no-sign-obst}. It can be viewed as an extension of the recent paper \cite{FigSer19} on the classical obstacle problem. Our central result can be found in Theorem~\ref{thm:main}.

We only consider singular points $x^0$ where the blow-up $p_{*,x^0}$ is a nonnegative polynomial. The nonnegativity assumption on $p_{*,x^0}$ seems imperative (and almost inevitable) for the Almgren's frequency formula (specificlly, the monotonicity of $r\mapsto\Phi(r,u(x^0+\cdot)-p_{*,x^0})$). For this reason, we introduce $\Sigma_{m}^+=\{x^0\in \Sigma_m:p_{*,x^0}\ge0\}$ for $1\le m\le n-2$, and study the structure of $\Sigma_m^+$ instead of $\Sigma_m$.


\subsection{Notation}\label{sec:notation}
For a point $x\in\R^n$ and $r>0$, by $B_r(x)$ we mean the ball of radius $r$ centered at $x$, i.e.,
$$
B_r(x):=\{y\in\R^n\,:\, |y-x|<r\}.
$$
We set \begin{align*}
&\cP:=\{\text{$p(x)$: homogeneous quadratic polynomial with $\Delta p=1$}\},\\
&\cP^+:=\{p\in\cP\,:\, p\ge0\}.
\end{align*}
For a solution $u$ of \eqref{eq:no-sign-obst}, we say that a free boundary point $x^0\in\Gamma(u):=\partial\Omega(u)\cap B_1$ is \emph{singular} if the rescalings $$
u_{x^0,r}(x):=\frac{u(x^0+rx)}{r^2}
$$
converges over a subsequence to $p_{*,x^0}\in \cP$ as $r\to0$. Such polynomial $p_{*,x^0}$ is called a \emph{blowup} of $u$ at $x^0$. The uniqueness of the blowup $p_{*,x^0}$ at each singular point $x^0$ can be found in \cite{CheFenLi22}. We denote by $\Sigma=\Sigma_u$ the set of singular points of $u$. We also write $L_{x^0}=\{x\in\R^n\,:\,p_{*,x^0}=|\D p_{*,x^0}|=0\}$.

For given function $v$, point $x^0$, radius $r>0$ and constant $\la>0$, we let 
\begin{align*}
    &\Phi(r,v,x^0):=\frac{r\int_{B_r(x^0)}|\D v|^2}{\int_{\partial B_r(x^0)}v^2}: \text{Almgren's frequency functional},\\
    &W_\la(r,v,x^0):=\frac1{r^{n-2+2\la}}\int_{B_r(x^0)}|\D v|^2-\frac{\la}{r^{n-1+2\la}}\int_{\partial B_r(x^0)}v^2:\text{Weiss-type energy functional},\\
    &H_\la(r,v,x^0):=\frac1{r^{n-1+2\la}}\int_{\partial B_r(x^0)}v^2:\text{Monneau's functional}.
\end{align*}

When $x^0\in \Sigma_u$ and $p_{*,x^0}\in \cP^+$ with $u\not\equiv p_{*,x^0}$, we define
$$
\la_{*,x^0}:=\Phi(0+,u(x^0+\cdot)-p_{*,x^0})=\lim_{r\to0+}\Phi(r,u(x^0+\cdot)-p_{*,x^0}).
$$
For the notational simplicity, if $x^0=0$ then we simply write $p_*=p_{*,x^0}$, $L=L_{x^0}$, $\Phi(r,v)=\Phi(r,v,x^0)$, $\la_{*}=\la_{*,0}$, etc. We refer Lemma~\ref{lem:mon-formula} for the existence of the limit of $\Phi$ as $r\to0+$.

When $0\in \Sigma$, we denote
\begin{align*}
    w(x):=u(x)-p_*(x),\quad w_r(x):=w(rx),\quad \tilde w_r(x):=\frac{w_r(x)}{\|w_r\|_{L^2(\partial B_1)}}.
\end{align*}

Next, we classify the singular set $\Sigma$ with the dimension of $L_{x^0}$
\begin{align*}
\Sigma=\cup_{m=0}^{n-1}\Sigma_{m}, \quad\text{where }\Sigma_m:=\{x^0\in \Sigma\,:\,\dim(L_{x^0})=m\}.
\end{align*}
For $m=n-1$, we denote for $\la>2$
$$
S_{n-1,\la}:=\{x^0\in \Sigma_{n-1}\,:\,\Phi(0+,u(x^0+\cdot)-p_{*,x^0})\ge\la\},
$$
and write $\Sigma_{n-1}$ as a disjoint union of its \qq{generic} part $\Sigma_{n-1,g}$ and \qq{anomalous} part $\Sigma_{n-1,a}$ defined as
\begin{align*}
    &\Sigma_{n-1,g}:=S_{n-1,5/2}=\{x^0\in \Sigma_{n-1}\,:\,\Phi(0+,u(x^0+\cdot)-p_{*,x^0})\ge5/2\},\\
    &\Sigma_{n-1,a}:=\Sigma_{n-1}\setminus S_{n-1,5/2}=\{x^0\in \Sigma_{n-1}\,:\,\Phi(0+,u(x^0+\cdot)-p_{*,x^0})<5/2\}.
\end{align*}
When $0\le m\le n-2$, we introduce the following notations 
\begin{align*}
    &\Sigma^+:=\{x^0\in \Sigma\,:\, p_{*,x^0}\ge0\},\\
    &\Sigma_m^+:=\{x^0\in \Sigma_m\,:\, p_{*,x^0}\ge0\},\\
    & S_\la^+:=\{x^0\in \Sigma^+\,:\, \Phi(0+,u(x^0+\cdot)-p_{*,x^0})\ge\la\},\\
    &S_{m,\la}^+:=\{x^0\in \Sigma_{m}^+\,:\, \Phi(0+,u(x^0+\cdot)-p_{*,x^0})\ge\la\},
\end{align*}
and write $\Sigma_m^+=\Sigma_{m,g}^+\cup \Sigma_{m,a}^+$, where
\begin{align*}
    &\Sigma^+_{m,g}:=S^+_{m,3}=\{x^0\in \Sigma^+_{m}\,:\,\Phi(0+,u(x^0+\cdot)-p_{*,x^0})\ge3\},\\
    &\Sigma^+_{m,a}:=\Sigma^+_m\setminus S_{m,3}^+=\{x^0\in \Sigma^+_{m}\,:\,\Phi(0+,u(x^0+\cdot)-p_{*,x^0})<3\}.
\end{align*}
We remark that for every $x^0\in \Sigma$, we have $p_{*,x^0}\ge0$ if and only if $\{u<0\}$ has zero-density near $x^0$, i.e., $\lim_{\rho\to0}\frac{|B_\rho(x^0)\cap \{u<0\}|}{|B_\rho|}=0$. 

We also notice that
$$
\Sigma^+=\Sigma_{n-1}\cup\left(\cup_{m=1}^{n-2}\Sigma_m^+\right).
$$

For a set $S\subset\R^n$, the maximal radius of $S$, $MR(S)$, is the radius of the largest ball contained in $S$. Given constants $M>0$ and $0<\e\le1$ and a modulus of continuity $\omega$, we say that a function $u$ belongs to $ P(M,\omega,\e,x^0)$ if
\begin{align*}
    &\, (1)\, \text{$u$ is a solution of \eqref{eq:no-sign-obst} in $B_1$},\\
    &\, (2)\,x^0\in \Sigma^+,\\
    &\,(3)\, \|D^2u\|_{L^\infty(B_1)}\le M,\\
    &\,(4)\, \frac{\sup_{y\in \overline{B_\e(x^0)}\cap \Gamma}MR(B_r(y)\cap \{u<0\})}r\le \omega(r),\quad 0\le r<1-\e.
\end{align*}

Throughout this paper, we assume $u\not\equiv p_{*,x^0}$ for any $x^0\in \Sigma$, as otherwise there is nothing to study $\Sigma_u=\Sigma_{p_{*,x^0}}$.

\subsection{Main results and structure of the paper}
The main result of this paper is as follows:

\begin{theorem}\label{thm:main}
Let $u$ be a solution of \eqref{eq:no-sign-obst} in $B_1$.
\begin{enumerate}[(a)]
  \item 
  For $n\ge 2$, $\Sigma_{n-1,g}$ is locally contained in a $C^{1,1/2}$ $(n-1)$-dimensional manifold, and $\Sigma_{n-1,a}$ is a relatively open subset of $\Sigma_{n-1}$.
    \item For $n\ge3$ and $1\le m\le n-2$, $\Sigma_{m,g}^+$ is locally contained in a $C^{1,1}$ $m$-dimensional manifold, and $\Sigma_{m,a}^+$ is a relatively open subset of $\Sigma_m^+$ with $\dim_{\cH}(\Sigma_{m,a}^+)\le m-1$.
    \item 
    Let $n\ge3$ and $1\le m\le n-2$. We additionally assume that for every $x^0\in \Sigma_m^+$, $u\in P(M,\omega,\e,x^0)$ for some $0<\e<1$ and a modulus of continuity $\omega(r)$ satisfying $\omega(r)\le C[\log(1/r)]^{-\frac1{2(n-1)}}$, $0<r<1$, for some $C>0$. Then $\Sigma_m^+$ is locally covered by a $C^{1,\log^{\frac1{2(n-1)}}}$ $m$-dimensional manifold.
\end{enumerate}
\end{theorem}

The rest of the paper is organized as follows. In Section~\ref{subsec:open-prob}, we present open problems on the no-sign obstacle problem and the superconductivity problem. In Section~\ref{sec:struc-sing-set}, we establish our main result, Theorem~\ref{thm:main}. In Section~\ref{sec:diss}, we discuss the structure of the top stratum $\Sigma_{n-1}$. Finally, in Appendix~\ref{appen:recur-ineq} we provide the proof of a recurrence inequality necessary to prove the main theorem.

\subsection{Open problems}\label{subsec:open-prob}

Here, we present several open problems that we believe may be of interest to the free boundary community and would require new ideas and approaches to address.

Obtaining a complete theory, parallel to Figalli-Serra, would naturally be a great achievement. This would require first of all to remove the small density assumption for the negative set of solutions, for lower stratum. For the highest stratum, this is obvious since blow up polynomial can not take  negative values, if their coincidence set has Hausdorff dimension $n-1$. 

An interesting problem encountered when trying the Figalli-Serra method is that the reduced problem, after linearization, is a thin obstacle problem, but without a sign condition  on the thin space; see Section~\ref{sec:diss}, in particular, \eqref{eq:Alm-blowup-nosign}. Currently, there are no theories developed for this problem.

One may also consider the so-called superconductivity problem, which is defined through the equation
\begin{align}\label{eq:super-con}
\Delta u = \chi_{\{ |\nabla u | >0\}} .
\end{align}
At singular points, even if the blow-up is a non-negative polynomial, new challenges arise. For example, for a solution $u$ of \eqref{eq:super-con} and a nonnegative blow-up $p_*$, $w:=u-p_*$ does not necessarily satisfy
$$
w\Delta w=0.
$$
This corresponds to Remark 2.2 in \cite{FigSer19} and plays a significant role in the proof of monotonicity formulas, Lemma~\ref{lem:mon-formula}.



\section{Structure of the singular set with zero-density of the negativity set}\label{sec:struc-sing-set}

In this paper we focus on understanding the structure of $\Sigma_{n-1}$ and $\Sigma^+_m$ for $1\le m\le n-2$, as the standard argument shows that $\Sigma_0$ is an isolated set. Indeed, we assume to the contrary that there is a sequence of free boundary points $x^k$ converging to a point $x^0\in \Sigma_0$. Then, we have for $r_k:=|x^k-x^0|$
$$
u_{r_k,x^k}(x)=\frac{u(x^k,r_kx)}{r_k^2}=\frac{u\left(x^0+r_k\left(x+\frac{x^k-x^0}{r_k}\right)\right)}{r_k^2}=u_{r_k,x^0}\left(x+\frac{x^k-x^0}{r_k}\right).
$$
Over a subsequence, $u_{r_k,x^k}$ converges in $C^1_{\loc}(\R^n)$ to a global solution $u_0$ and $\frac{x^k-x^0}{r_k}$ converges to a point $z^0\in\partial B_{1}$. Thus,
$$
u_0(x)=p_*(x+z^0),\quad x\in\R^n.
$$
Then, from $u_{r_k,x^k}(x)=|\D u_{r_k,x^k}(x)|=0$, we find $p_*(z^0)=|\D p_*(z^0)|=0$. This is a contradiction since $p_*$ is the $n$-dimensional homogeneous polynomial.

\begin{lemma}[Monotonicity formulas]\label{lem:mon-formula}
Let $u$ be a solution of \eqref{eq:no-sign-obst} and $p\in \cP^+$ with $0\in \Sigma$ and $u\not\equiv p$. Then
\begin{enumerate}
    \item $\Phi(r,u-p)$ is nondecreasing in $r\in (0,1)$, and $\Phi(0+,u-p)\ge2$.
    \item For any $\la\ge0$, $W_\la(r,u-p)$ is nondecreasing in $r\in (0,1)$.
    \item For $0\le\la\le\Phi(0+,u-p)$, $H_\la(r,u-p)$ is nondecreasing in $r\in (0,1)$. On the other hand, when $\la>\Phi(0+,u-p)$, we have $H_\lambda(r,u-p)\to \infty$ as $r\to 0$.
\end{enumerate}
\end{lemma}

We remark that $p\in \cP^+$ in Lemma~\ref{lem:mon-formula} is not necessarily a blowup of $u$ at 0.

\begin{proof}
Using that $u$ is a solution of \eqref{eq:no-sign-obst} and $p\in \mathcal{P}^+$, one can easily show that $w\Delta w=p\rchi_{\{u=0\}}\ge0$. With this property of $w$ at hand, we can repeat the argument in \cite{FigSer19} to conclude the lemma.
\end{proof}

To proceed, we state two lemmas that serve as analogues of \cite[Lemma~2.11]{FigSer19} and \cite[Lemma~2.12]{FigSer19}, whose proofs follow in the same way.

\begin{lemma}\label{lem:blowup-ineq}
For $\tilde w_r$ as in Proposition~\ref{prop:blow-up-struc}, we assume that $\tilde w_{r_{k_l}}\to q$ weakly in $W^{1,2}(B_1)$ for some sequence $r_{k_l}\searrow0$. Then 
\begin{align}
    \label{eq:auxi-Monneau}
    \int_{\partial B_1}q(p_*-p)\ge0\quad\text{for all }p\in \cP^+.
\end{align}
\end{lemma}

\begin{lemma}\label{lem:2-hom-har-poly}
    Let $p_*\in \cP^+$, and suppose that $q\not\equiv0$ is a $2$-homogeneous harmonic polynomial satisfying \eqref{eq:auxi-Monneau}. Then, in an appropriate system of coordinates, \eqref{eq:blow-up-struc-matrix} holds.
\end{lemma}

With Lemmas~\ref{lem:blowup-ineq} and \ref{lem:2-hom-har-poly} at hand, we can obtain the following result, which is an analogue of \cite[Proposition~2.10 (a)]{FigSer19}. The proof proceeds in the same manner as that of \cite[Proposition~2.10 (a)]{FigSer19}.

\begin{proposition}
    \label{prop:blow-up-struc}
Let $u$ be a solution of \eqref{eq:no-sign-obst} with $0\in \Sigma^+$. Suppose $m\in \{0,1,2,\cdots,n-2\}$ is the dimension of $L=\{p_*=0\}$. Then we have $\lambda_*\in\{2,3,4,5,\cdots\}$. Moreover, for every sequence $r_k\searrow0$, there exists a subsequence $r_{k_l}$ such that $\tilde w_{r_{k_l}}\to q$ weakly in $W^{1,2}(B_1)$, where $q\not\equiv0$ is a $\lambda_*$-homogeneous harmonic polynomial. In addition, if $\lambda_*=2$, then in an appropriate coordinate frame it holds 
\begin{align}\label{eq:blow-up-struc-matrix}
    D^2p_*=\left(\begin{matrix}  
\begin{array}{c|c}  
  \begin{matrix}
      \mu_1& & \\
      &\ddots& \\
      && \mu_{n-m}
  \end{matrix} & 0_m^{n-m} \\   
  \hline  
  0_{n-m}^m & 0_m^m  
 \end{array}  
\end{matrix}  \right)\quad\text{and}\quad D^2q=   \left(\begin{matrix}  
\begin{array}{c|c}  
  \begin{matrix}
      t& & \\
      &\ddots& \\
      && t
  \end{matrix} & 0_m^{n-m} \\   
  \hline  
  0_{n-m}^m & -N 
 \end{array}  
\end{matrix}  \right),
\end{align}
where $\mu_1,\cdots,\mu_{n-m},t>0$, $\Sigma_{i=1}^{n-m}\mu_i=1$, and $N$ is a symmetric $m\times m$ matrix with $tr\,N=(n-m)t>0$.


\end{proposition}

\begin{remark}
    \cite[Proposition~2.10 (a)]{FigSer19} states that the matrix $N$ in \eqref{eq:blow-up-struc-matrix} is both nonnegative definite and has a positive trace. However, this is a minor mistake, and only the positivity of the trace can be guaranteed, as corrected in \cite{FigRosSer20}. We also stress that the trace of $N$ is still positive in our "no-sign" obstacle problem, thanks to the nonnegativity condition on the blow-up $p_*$.
\end{remark}

\begin{remark}\label{rem:n-1}
In Proposition~\ref{prop:blow-up-struc}, we only consider the case $m\le n-2$. However, in the classical obstacle problem \cite{FigSer19}, the authors obtained the following results when $m=n-1$: for $m=n-1$, $\la_*\ge 2+\al_0$ for some $\al_0=\al_0(n)$. Moreover, for every sequence $r_k\searrow0$, there exists a subsequence $r_{k_l}$ such that $\tilde w_{r_{k_l}}\to q$ weakly in $W^{1,2}(B_1)$, where $q\not\equiv0$ is a $\la_*$-homogeneous solution of
\begin{align}\label{eq:(n-1)-hom-Alm-blowup}
    \begin{cases}
        \Delta q\le0,\,\,\, q\Delta q=0\quad\text{in }\R^n,\\
        \Delta q=0\quad\text{in }\R^n\setminus L,\\
        q\ge0\quad\text{on }L.
    \end{cases}
\end{align}
They obtained the lower bound $\la_*\ge 2+\al_0$ by using \eqref{eq:blow-up-struc-matrix} and \eqref{eq:(n-1)-hom-Alm-blowup}, which plays a significant role in establishing the structure of the top stratum $\Sigma_{n-1}$.

In our case, once we have $\la_*\ge 2+\al_0$, then we can follow the argument in \cite{FigSer19} to obtain the structure of $\Sigma_{n-1}$; see Conjecture~\ref{conj:n-1}. Compared to \cite{FigSer19}, the main difficulty arising in our case is that we do not have $q\ge0$ on $L$ in the equation \eqref{eq:(n-1)-hom-Alm-blowup}. Without the nonnegativity of $q$ on $L$, the blowup $q$ in our case is not a solution of the thin obstacle problem.
\end{remark}




\begin{lemma}\label{lem:y_infty}
For $n\ge 3$, suppose $0\in \Sigma^+$ and $m=\dim(L)\le n-2$. Asume that there exist a sequence of points $x_k\in\Sigma$ converging to $0$ and a sequence of radii $r_k\searrow 0$ with $|x_k|\le r_k/2$ such that
\begin{align*}
    &\tilde w_{r_k}=\frac{(u-p_*)(r_k\cdot)}{\|(u-p_*)(r_k\cdot)\|_{L^2(\partial B_1)}}\to q\quad\text{weakly in }W^{1,2}(B_1),\\
    &y_k:=\frac{x_k}{r_k}\to y_\infty.
\end{align*}
Then $y_\infty\in L$ and $q(y_\infty)=0$.
\end{lemma}

\begin{proof}
    The proof follows from Almgren- and Monneau- monotonicity formulas applied to $u(x_k+\cdot)-p_{*}$; see \cite[Lemma~3.2]{FigSer19}.
\end{proof}

\begin{lemma}\label{lem:anom-1-discrete}
    If $n\ge3$, then $\Sigma_{1,a}^+$ is a discrete set.
\end{lemma}

\begin{proof}
With Lemma~\ref{lem:y_infty} at hand, we can argue as in \cite[Lemma~3.4]{FigSer19} to obtain Lemma~\ref{lem:anom-1-discrete}.
\end{proof}

\begin{lemma}\label{lem:anom-dim}
Suppose $n\ge4$ and $2\le m\le n-2$. Then $\dim_{\mathcal{H}}(\Sigma^+_{m,a})\le m-1$.    
\end{lemma}

\begin{proof}
The proof is similar to that of \cite[Lemma~3.6]{FigSer19}.    
\end{proof}

To proceed, we need the following version of Whiteney's Extension Theorem. See e.g., \cite{Fef09}.

\begin{lemma}[Whitney's Extension Theorem]\label{lem:Whitney}
Let $0<\be\le1$, $l\in \mathbb{N}$, $K\subset\R^n$ a compact set, and $f:K\to\R$ a given function. Suppose that for every $x^0\in K$, there exists a polynomial $P_{x^0}$ of degree $l$ such that
\begin{enumerate}[(i)]
    \item $P_{x^0}(x^0)=f(x^0)$,
    \item $|D^kP_{x^0}(x^0)-D^kP_{x^1}(x^0)|\le C|x^0-x^1|^{l+\be-k}$ for all $x\in K$ and $k\in\{0,1,\cdots,l\}$, where $C>0$ is a constant independent of $x^0$.
\end{enumerate}
Then, there exists a function $F\in C^{l,\be}(\R^n)$ such that 
$$
F|_K=f\quad\text{and}\quad F(x)=P_{x^0}(x)+O\left(|x-x^0|^{l+\be}\right) \,\,\,\text{for any }x^0\in K. 
$$
\end{lemma}

Next, we prove the structure of $S_{m,\la}^+$. The proof is similar to that of \cite[Lemma~3.1]{FigSer19}, but we need some modifications as we do not have the closedness of $\Sigma^+$ and $S_\la^+$.

\begin{lemma}\label{lem:S-class}
Let $n\ge2$, $1\le m\le n-1$ and $\la>2$ be given. Let $l\in\mathbb{N}$ and $0<\be\le1$ be such that $l+\be=\la$. Then $S_{m,\la}^+$ is locally contained in an $m$-dimensional manifold of class $C^{l-1,\be}$.
\end{lemma}

\begin{proof}
We define 
$$
P_{x^0}(x):=p_{*,x^0}(x-x^0),\quad x^0\in S_\la^+\cap B_{1/2}.
$$
For such $x^0$, we have by Monneau's monotonicity formula
$$
\|u(x^0+\rho\cdot)-p_{*,x^0}(\rho\cdot)\|_{L^2(\partial B_1)}\le (2\rho)^\la\left\|u\left(x^0+\frac12\cdot\right)-p_{*,x^0}\left(\frac12\cdot\right)\right\|_{L^\infty(\partial B_1)},\quad 0<\rho\le 1/2.
$$
Integrating in $\rho$ gives
$$
\|u(x^0+\rho\cdot)-p_{*,x^0}(\rho\cdot)\|_{L^2(B_1)}\le C\rho^\la,\quad 0<\rho\le1/2.
$$
This corresponds to equation~(3.21) in \cite{FigSer19}. With that estimate at hand, we can repeat the argument in \cite{FigSer19} to get
\begin{align}
    \label{eq:P-diff-est}
    |D^kP_{x^0}(x^0)-D^kP_{x^1}(x^0)|\le C|x^0-x^1|^{l+\be-k}
\end{align}
for all $x^0,x^1\in S_\la^+\cap B_{1/2}$ and $k\in\{0,1,\cdots,l\}$.

We considered $P_{x^0}$ for $x^0$ contained in $S^+_\la$ above. Now we extend the range of $x^0$ from $S_\la^+$ to $\overline{S_\la^+}$. For a compact set $K:=\overline{S^+_\la}\cap \overline{B_{1/4}}$, we define $P_{x^0}$ for $x^0\in K$ as
$$
    P_{x^0}(x):=\begin{cases}
    p_{*,x^0}(x-x^0),& x^0\in S^+_\la\cap\overline{B_{1/4}},\\
    \lim_{S_\la^+\ni x^j\to x^0}P_{x^j}(x),& x^0\in (\overline{S^+_\la}\setminus S^+_\la)\cap\overline{B_{1/4}}.
    \end{cases}
$$
Here, $P_{x^0}$ is well-defined for $x^0\in (\overline{S^+_\la}\setminus S^+_\la)\cap\overline{B_{1/4}}$ thanks to \eqref{eq:P-diff-est}. Indeed, from the condition $l+\be=\la>2$, we find $l\ge2$. Thus, if we write $p_{*,x^0}(x)=\frac12\mean{A_{x^0}x,x}$ for every $x^0\in S_\la^+$, then the estimate \eqref{eq:P-diff-est} applied with $k=2$ gives $|A_{x^0}-A_{x^1}|\le C|x^0-x^1|^{l+\be-2}$ for any $x^0,x^1\in S_\la^+\cap B_{1/2}$. This implies that $S_\la^+\cap \overline{B_{1/4}}\ni x^j\mapsto P_{x^j}=p_{*,x^j}(\cdot-x^j)$ is a Cauchy sequence, hence convergent to a quadratic polynomial.

We observe that $P_{x^0}(x^0)=0$ for every $x^0\in \overline{S^+_\la}\cap \overline{B_{1/4}}$ and that \eqref{eq:P-diff-est} holds for any $x^0,x^1\in \overline{S_\la^+}\cap\overline{B_{1/4}}$ and $0\le k\le l$ since it holds for any $x^0,x^1\in \overline{S_\la^+}\cap\overline{B_{1/4}}$ and $k\le 2$ (based on the argument above) and the norm $\|\cdot\|_{L^\infty(B_1)}$ is equivalent to the norm $\|\cdot\|_{C^l(B_1)}$ on the space of quadratic polynomials. Then $K$, $\{P_{x^0}\}_{x^0\in K}$, $f\equiv 0$ satisfy the assumptions of Whitney's Extension Theorem (Lemma~\ref{lem:Whitney}), and hence there exists a function $F\in C^{l,\be}(\R^n)$ such that 
$$
F(x)=P_{x^0}(x)+O\left(|x-x^0|^{l+\be}\right),\quad x^0\in K.
$$
In particular, we have
$$
F(x)=p_{*,x^0}(x-x^0)+O\left(|x-x^0|^{l+\be}\right)\quad\text{for }x^0\in S^+_\la\cap \overline{B_{1/4}}.
$$
Thus $$
S^+_{m,\la}\cap B_{1/4}\subset K\subset \{\D F=0\}.
$$
As $\dim(\text{Ker}(D^2F(x^0)))=\dim(\text{Ker}(D^2p_{*,x^0}(0)))=m$ for every $x^0\in S^+_{m,\la}\cap B_{1/4}$, we find that, after taking a change of coordinates if necessary, $\text{Rank}(D^2_{(x_1,\cdots,x_{n-m})}F(x^0))=n-m$ is maximal. Now, Lemma~\ref{lem:S-class} follows by the Implicit Function Theorem.
\end{proof}

Now we prove our central result, Theorem~\ref{thm:main}.

\begin{proof}[Proof of Theorem~\ref{thm:main}]
We first prove $(a)$. As $\Sigma_{n-1,g}=S_{n-1,5/2}$, its $C^{1,1/2}$ structure follows from Lemma~\ref{lem:S-class}. Since $x^0\longmapsto \Phi(0+,u(x^0+\cdot)-p_{*,x^0})$ is upper semicontinuous, $\Sigma_{n-1,g}=S_{n-1,5/2}=\{x^0\in \Sigma_{n-1}\,:\, \Phi(0+,u(x^0+\cdot)-p_{*,x^0}\ge5/2\}$ is relatively closed in $\Sigma_{n-1}$, and thus $\Sigma_{n-1,a}$ is relatively open. 

Next, we prove $(b)$. As $\Sigma_{m,g}^+=S^+_{m,3}$, its $C^{1,1}$ structure follows from Lemma~\ref{lem:S-class}. The relative openness of $\Sigma_{m,a}^+$ can be obtained as in $(a)$. The remaining part in $(b)$ can be found in Lemma~\ref{lem:anom-1-discrete} and Lemma~\ref{lem:anom-dim}.

Finally, we prove $(c)$. It can be obtained in a similar way as in the proof of Theorem~1.1 $(e)$ in \cite{FigSer19}. The only nontrivial part is Caffarelli's asymptotic convexity estimate, the lower bound of the second derivatives of solutions; see equation~(3.24) in \cite{FigSer19}. 

We now show that the similar estimate holds in our case under the additional assumption $u\in P(M,\omega,\e,x^0)$. Specifically, we claim that for any $r>0$ small enough
\begin{align}
    \label{eq:second-deriv-bound}
    D_{ii}u\ge -C[\log(1/r)]^{-\be} \quad\text{in }B_r(x^0),\quad x^0\in \Sigma_m^+,\quad 1\le i\le n,
\end{align}
where $\be=\frac1{2(n-1)}$. To prove the claim, we follow the idea in the proof of Lemma~4.5 in \cite{KarSha99}. Let 
$$
M_k:=\sup\{-\inf_{B_{2^{-k}}(y)}D_{ii}u\,:\, y\in \Gamma\cap\overline{B_{\e/4}}\},\quad k\ge1.
$$
Note that $M_{k+1}\le M_k\le M$. Let $k_0$ be the smallest integer such that $2^{-(k_0+1)}\le \e/4$. For $k\ge k_0$, by applying \cite[Lemma~4.6]{KarSha99} (with $\rho=2^{-(k+1)}$ and $\delta=\delta_k\in(0,1)$ to be chosen later) and Harnack Inequality, we can obtain
$$
D_{ii}u+M_k\ge C\delta_k^{n-1}(M_k-C\sqrt{\delta_k}-C\omega(2^{-k}))\quad\text{in }B_{2^{-(k+1)}}(y).
$$
This implies
$$
-M_{k+1}+M_k\ge C\delta_k^{n-1}(M_k-\tilde C\sqrt{\delta_k}-C\omega(2^{-k})).
$$
Taking $\delta_k=\e_0 M_k^2$ with $\e_0>0$ small enough so that $\delta_k<1$ and $M_k>(\tilde C+1)\sqrt{\delta_k}$, we get 
$$
-M_{k+1}+M_k\ge CM_k^{2n-1}-CM_k^{2n-2}\omega(2^{-k}).
$$
As the logarithmic decay rate of $\omega$ gives $\omega(2^{-k})\le C(1/k)^\be$, we further have
$$
-M_{k+1}+M_k\ge CM_k^{2n-1}-CM_k^{2n-2}(1/k)^\be,\quad k\ge k_0.
$$
Then, Lemma~\ref{lem:M_k-est} in Appendix gives that for some $K_0\in \mathbb{N}$ and $C_0>0$
$$
M_k\le C_0(1/k)^\be,\quad k\ge K_0.
$$
This implies \eqref{eq:second-deriv-bound} for small $r>0$, and concludes $(c)$.
\end{proof}


\section{Discussion (Structure of the top stratum)}\label{sec:diss}
In Theorem~\ref{thm:main} (a), we established the structure of $\Sigma_{n-1,g}$ and showed that $\Sigma_{n-1,a}$ is relatively open. Additionally, we propose the following conjecture regarding the structure of the top stratum $\Sigma_{n-1}$.
\begin{conjecture}\label{conj:n-1}
    Let $u$ be a solution of \eqref{eq:no-sign-obst} in $B_1$.
    \begin{enumerate}[(a)]
    \item For $n=2$, $\Sigma_{1}$ is locally contained in a $C^{1,1/2}$ curve.
  \item For $n\ge 3$, $\Sigma_{n-1,a}$ is a discrete set if $n=3$, while $\dim_{\cH}(\Sigma_{n-1,a})\le n-3$ when $n\ge 4$.
  \item For $n\ge3$, $\Sigma_{n-1}$ is locally covered by a $C^{1,\al_0}$ $(n-1)$-dimensional manifold for some $\al_0=\al_0(n)\in(0,1)$.
  \end{enumerate}
\end{conjecture}
An analogue of Conjecture~\ref{conj:n-1} was obtained in \cite{FigSer19} for the classical obstacle problem. A key to proving this result was showing that for any point $x^0\in \Sigma_{n-1}$, $\la_{*,x^0}\ge 2+\al_0$ for some dimensional constant $\al_0\in (0,1)$. However, as mentioned in Remark~\ref{rem:n-1}, a different approach seems necessary to derive the lower bound for $\la_{*,x^0}$ in our case. Once this bound is established, we can follow the approach in \cite{FigSer19} to achieve Conjecture~\ref{conj:n-1}.

Below, we prove Conjecture~\ref{conj:n-1} under the assumption 
\begin{align}
    \label{eq:assump}
    \la_{*,x^0}\ge2+\al_0,\quad x^0\in\Sigma_{n-1}.
\end{align}
Note that if \eqref{eq:assump} holds, then $\Sigma_{n-1}=S_{n-1,2+\al_0}$, thus Conjecture~\ref{conj:n-1} (c) follows from Lemma~\ref{lem:S-class}. To treat (a) and (b) in the conjecture, we use the following result in \cite{FocSpa18} regarding the thin obstacle problem:

\begin{proposition}
    \label{prop:sig-2-dim-hom-sol}
Let $v:\R^2\to\R$ be $\la$-homogeneous and even-symmetric in $x_2$, and assume that $v$ is a weak solution to 
\begin{align}\label{eq:har}
\Delta v=0\quad\text{in }\R^2\setminus(\{x_2=0\}\cap\{v=0\}).
\end{align}
Then, one of the following holds:
\begin{enumerate}[(i)]
    \item $\{x_2=0\}\cap \{v=0\}=\{x_1=x_2=0\}$, $\la\in \mathbb{N}$ and $v$ is a multiple of $\Phi_m$;
    \item $\{x_2=0\}\cap \{v=0\}=\{x_1\le0,x_2=0\}$ ($\{x_2=0\}\cap\{v=0\}=\{x_1\ge0,x_2=0\}$, respectively), $\la=m+1/2$ for some $m\in \mathbb{N}\cup\{0\}$ and $v$ is a multiple of $\Psi_m$ ($\Psi_m(-x_1,x_2)$, respectively);
    \item $\{x_2=0\}\cap \{v=0\}=\{x_2=0\}$, $\la=m+1$ for some $m\in \mathbb{N}\cup\{0\}$ and $v$ is a multiple of $\Pi_m$.
\end{enumerate}
In addition, if $v$ is a solution of the thin obstacle problem (i.e., $v$ satisfies \eqref{eq:har}, $v\ge0$ on $\{x_2=0\}$ and $\Delta v\le0$ in $\R^2$), then $m$ is even in $(i)$ and $(iii)$, and $m$ is odd in $(ii)$.
\end{proposition}

In the above proposition, $\Phi_m$, $\Psi_m$ and $\Pi_m$ are $m$-homogeneous solutions of \eqref{eq:har} (see equations (A.2)-(A.4) in \cite{FocSpa18}) whose explicit expressions are not our interest.

As pointed out in \cite{FigSer19}, Proposition~\ref{prop:sig-2-dim-hom-sol} implies that if $v$ is an even-symmetric homogeneous solution to the thin obstacle problem in $\R^2$, then its homogeneity belongs to the set
$$
\{1,2,3,4,\cdots\}\cup\{3/2,7/2,11/2,15/2,\cdots\}.
$$

Let $q:\R^2\to\R$ be an even-symmetric homogeneous solution of 
\begin{align}
    \label{eq:Alm-blowup-nosign}
    \begin{cases}
        \Delta q\le0,\,\,\, q\Delta q=0\quad\text{in }\R^n,\\
        \Delta q=0\quad\text{in }\R^n\setminus L,
    \end{cases}
\end{align}
where $L$ is a line. Note that $q$ can be seen as a solution of the thin obstacle problem without the sign condition on $L$. One can easily check from the proof of Proposition~\ref{prop:sig-2-dim-hom-sol} that the sign condition on $v$ is used only to obtain that $m$ is even in $(i)$ and $m$ is odd in $(ii)$. This observation gives the following lemma concerning $q$.

\begin{lemma}
    \label{lem:2-dim-sol-hom}
Let $q$ be as in \eqref{eq:Alm-blowup-nosign}. Then its possible homogeneity is contained in the set 
$$
\{k/2\,:\, k\in \mathbb{N}\}.
$$
\end{lemma}

Lemma~\ref{lem:2-dim-sol-hom} readily gives the following classification of possible Almgren frequencies when the dimension $n=2$.

\begin{lemma}\label{lem:classi-hom-(n=2)}
Let $n=2$ and $0\in\Sigma_1$. If \eqref{eq:assump} holds, then $\la_*=\Phi(0+,u-p_*)$ belongs to the set 
$$
\{k/2\,:\, k=5,6,7,8,\cdots\}.
$$
\end{lemma}

\begin{proof}
By following the argument in \cite[Proposition~2.10]{FigSer19}, we can obtain that every blowup $q$ of $u-p_*$ is a $\la_*$-homogeneous solution to \eqref{eq:Alm-blowup-nosign} with $\la_*\ge2+\al_0>2$. We decompose $q$ into the sum of even and odd functions with respect to $L$
$$
q=q_{even}+q_{odd}.
$$
Note that $q_{even}$ and $q_{odd}$ are homogeneous of degree $\la_*$. As $q$ is harmonic in $\R^2\setminus L$, one can easily show that $q_{odd}$ is harmonic in $\R^2$. Thus, its homogeneity is contained in the set 
$$
\{1,2,3,4,\cdots\}.
$$
On the other hand, since $q_{even}$ is an even-symmetric solution to \eqref{eq:Alm-blowup-nosign}, we have by Lemma~\ref{lem:2-dim-sol-hom} that the possible homogeneity of $q_{even}$ belongs to the set 
$$
\{k/2\,:\, k=1,2,3,4,5,\cdots\}.
$$
Therefore,
$$
\la_*\in \{k/2\,:\, k=1,2,3,4,5,\cdots\},
$$
which combined with $\la_*>2$ concludes the lemma.    
\end{proof}

\begin{remark}\label{rem:hom-classi}
Lemma~\ref{lem:classi-hom-(n=2)} is an analogue of \cite[Lemma~3.1]{FigSer19} for the classical obstacle problem, where the authors proved that $\la_*=\Phi(0+,u-p_*)$ belongs to
$$
\{3,4,5,6,\cdots\}\cup\{7/2, 11/2, 15/2, 19/2,\cdots\}.
$$
Lemma~\ref{lem:classi-hom-(n=2)} says that in our no-sign case $\la_*$  can be contained in a larger set including additional values
$$
5/2,9/2,13/2,17/2,\cdots.
$$
Now we show that this values are not redundant by providing an explicit example.

We consider solutions of \eqref{eq:no-sign-obst} defined locally near the origin with polar coordinates by
$$
u(x)=\frac12x_2^2-\frac1{1+\mu/2}r^{1+\mu/2}\sin(1+\mu/2)\theta+\cdots,\quad \mu=4k+3\, (k=0,1,2,3,\cdots),
$$
see Equation~5.3 in \cite{KinNir77} (also \cite{CafSha04}). From $1+\mu/2>2$, the $2$-homogeneous blowup of $u$ at $0$ should be a polymial $p_*(x)=\frac12x_2^2$. One can also see that the cusp appears near the free boundary $0\in \Gamma$, and thus $0\in \Sigma^+_1$. A direct computation shows 
$$
\tilde w_r(x)\approx -r^{1+\mu/2}\sin(1+\mu/2)\theta+\cdots.
$$
Therefore, $\tilde w_r(x)$ converges to a harmonic function $q(x)=-Cr^{1+\mu/2}\sin(1+\mu/2)\theta$ for some constant $C>0$, and hence 
$$
\Phi(0+,u-p_*)=\Phi(q,1)=1+\mu/2\in\{5/2,9/2,13/2,17/2,\cdots\}.
$$
\end{remark}

Now we are ready to prove Conjecture~\ref{conj:n-1} (a) when \eqref{eq:assump} holds.

\begin{proof}[Proof of Conjecture~\ref{conj:n-1} (a)]
By applying Lemma~\ref{lem:classi-hom-(n=2)}, we can obtain $\Sigma_1=S_{1,5/2}$. Thus, Conjecture~\ref{conj:n-1} $(a)$ follows by applying Lemma~\ref{lem:S-class} with $\la=5/2$.
\end{proof}

To obtain Conjecture~\ref{conj:n-1} (b), we consider the possible Almgren frequencies when $n\ge3$.

\begin{theorem}
    \label{thm:(n-1)-hom-Alm-freq}
Let $q$ be a solution of \eqref{eq:Alm-blowup-nosign} in $B_1$. Under the assumption \eqref{eq:assump}, for all $z$ in the coincidence set $\{q=0\}\subset L$, it holds
$$
\Phi(0+,q(z+\cdot))\in \{k/2\,:\, k=2,3,4,5,\cdots\}
$$
except for a set of Hausdorff dimension at most $n-3$.    
\end{theorem}

\begin{proof}
Theorem~\ref{thm:(n-1)-hom-Alm-freq} is a part of \cite[Theorem~1.3]{FocSpa18} concerning the thin obstacle problem. With Lemma~\ref{lem:2-dim-sol-hom} at hand, we can obtain the theorem by following the argument in \cite{FocSpa18}, in particular Lemma~5.3 and Section~8.1 therein, with larger set of possible frequencies $\{k/2\,:\, k=2,3,4,5,\cdots\}$.    
\end{proof}

\begin{lemma}\label{lem:Alm-blowup-accumu-(n-1)}
For $n\ge3$, let $0\in \Sigma_{n-1}$ with $\la_*\ge 2+\al_0$. Suppose that there exist a sequence of points $x_k\in \Sigma_{n-1}$ converging to $0$ and a sequence of radii $r_k\searrow0$ with $|x_k|\le r_k/2$ such that
\begin{align*}
    &\tilde w_{r_k}=\frac{(u-p_*)(r_k\cdot)}{\|(u-p_*)(r_k\cdot)\|_{L^2(\partial B_1)}}\to q\quad\text{weakly in }W^{1,2}(B_1),\\
    &y_k:=\frac{x_k}{r_k}\to y_\infty.
\end{align*}
Let $e\in S^{n-1}\cap L^\perp$, and denote by $q^{\text{even}}$ and $q^{\text{odd}}$ the even odd parts of $q$ with respect to $L$, respectively, i.e.,
\begin{align*}
    q^{\text{even}}(x)=\frac{q(x)+q(x-2(e\cdot x)e)}2,\quad q^{\text{odd}}(x)=\frac{q(x)-q(x-2(e\cdot x)e)}2.
\end{align*}
Then $y_\infty\in L$, and for $\la:=\inf_k\{\la_{*,x_k}\}\ge2+\al_0$
$$
\frac1{\rho^{2\la}}\dashint_{\partial B_\rho} q^{\text{even}}(y_\infty+x)^2\le 2^{2\la}\dashint_{\partial B_{1/2}}q^{\text{even}}(y_\infty+x)^2,\quad 0<\rho<1/2.
$$
Moreover, if $\la_*$ is not an integer, then 
$$
\frac1{\rho^{2\la}}\dashint_{\partial B_\rho} q(y_\infty+x)^2\le 2^{2\la}\dashint_{\partial B_{1/2}}q(y_\infty+x)^2,\quad 0<\rho<1/2.
$$
\end{lemma}

\begin{proof}
The proof follows from Almgren- and Monneau monotonicity formulas applied to $u(x_k+\cdot)-p_{*,x_k}$ and the fact that $\Phi(0+,u(x_k+\cdot)-p_{*,x_k})\ge 2+\al_0$; see \cite[Lemma~3.3]{FigSer19}.
\end{proof}

\begin{lemma}\label{lem:anom-discrete-(n=3)}
For $n=3$, if \eqref{eq:assump} holds, then $\Sigma_{n-1,a}$ is a discrete set.
\end{lemma}

\begin{proof}
We follow the idea in the proof of \cite[Lemma~3.7]{FigSer19}.

Towards a contradiction we assume there is a sequence of points $x_k\in \Sigma_{n-1}$ such that $x_k\to 0\in \Sigma_{n-1,a}$. We define $r_k:=2|x_k|\to0$. Then, over a subsequence $\tilde w_{r_k}\to q$ and $\frac{x_k}{r_k}\to z\in \partial B_{1/2}$, where $q$ is a $\la_*$-homogeneous solution of \eqref{eq:Alm-blowup-nosign}. Note that $q$ satisfies Almgren's frequency formula and Monneau Monotonicity formula since it satisfies $q\Delta q=0$.
From $0\in \Sigma_{n-1,a}$, we have $2+\al_0\le \la_*<5/2$, thus $\la_*$ is not an integer. Then Lemma~\ref{lem:Alm-blowup-accumu-(n-1)} gives $z\in L$ and
$$
\frac1{\rho^{2(2+\al_0)}}\dashint_{\partial B_\rho(z)}q^2\le C,\quad 0<\rho<1/2
$$
for some constant $C>0$ independent of $\rho$.
This implies that $q(z)=|\D q(z)|=|D^2q(z)|=0$ and that $\la^z:=\Phi(0+,q(z+\cdot))\ge 2+\al_0$. Let $q^z$ be the Almgren blowup of $q$ at $z$, i.e., $q^z(x)=\lim_{r_j\to0}\frac{q(z+r_jx)}{\|q(z+r_j\cdot)\|_{L^2(\partial B_1)}}$. Then it is a $\la^z$-homogeneous solution of \eqref{eq:Alm-blowup-nosign} and has translation symmetry in the direction $z$. Since $n=3$, $q^z$ depends only on two variables. We then have by Lemma~\ref{lem:2-dim-sol-hom} that
$$
\la^z\in\{k/2\,:\, k=1,2,3,4,\cdots\}.
$$
This, along with $\la^z\ge 2+\al_0$, gives $\la^z\ge 5/2$. We now apply Almgren's frequency formula and the fact that $q$ is homogeneous (with respect to the origin) to get a contradiction
\begin{equation*}
    5/2\le\la^z=\Phi(0+,q(z+\cdot))\le \Phi(\infty,q(z+\cdot))=\Phi(\infty,q)=\Phi(0+,q)=\la_*<5/2.\qedhere
\end{equation*}
\end{proof}

When $n\ge4$, we can prove the following lemma by applying Theorem~\ref{thm:(n-1)-hom-Alm-freq} and repeating the argument in \cite[Lemma~3.9]{FigSer19}.

\begin{lemma} \label{lem:anom-(n-1)-dim}If $n\ge4$ and \eqref{eq:assump} holds, then $\dim_{\cH}(\Sigma_{n-1,a})\le n-3$.
\end{lemma}

\begin{proof}
Suppose that the lemma is not true. Then $\cH^\be_\infty(\Sigma_{n-1,a})>0$ for some $\be>n-3$. Without loss of generality, we may assume $0\in \Sigma_{n-1,a}$ is a density point, i.e., for some sequence $r_K\searrow0$,
$$
\lim_{k\to\infty}\frac{\cH_\infty^\be(\Sigma_{n-1,a}\cap B_{r_k})}{r_k^\be}\ge c_{n,\be}>0.
$$
We then consider the \qq{accumulation set} for $A_{n,a}$ along $r_k$ at $0$
\begin{multline*}
    \cA=\cA_{\Sigma_{n-1,a},\{r_k\}}:=\{z\in\overline{B_{1/2}}\,:\, \text{there exist a sequence of points $\{z_l\}_{l\ge1}$}
    \\ \text{and a sequence of numbers $\{k_l\}_{l\ge1}$ such that $z_l\in r_{k_l}^{-1}\Sigma_{n-1,a}\cap B_{1/2}$ and $z_l\to z$}\}.
\end{multline*}
Then $\cH^\be_\infty(\cA)>0$ (see e.g. Lemma~3.5 in \cite{FigSer19}). Over a subsequence $w_{r_k}\to q$ weakly in $B_1$, where $q$ is a $\la_*$-homogeneous solution of \eqref{eq:Alm-blowup-nosign}. From $0\in \Sigma_{n-1,a}$ and \eqref{eq:assump}, we have $2+\al_0\le \la_*<5/2$. Let 
$$
S:=\{z\in\overline{B_1}\cap \{q=0\}\cap L\,:\, \Phi(0+,q(z+\cdot))\ge 2+\al_0\}.
$$
Then the definition of $\cA$ implies $\cA\subset S$ (see the proof of Lemma~\ref{lem:anom-discrete-(n=3)}). Thus,
$$
0<\cH_\infty^\be(\cA)\le \cH^\be_\infty(S).
$$
We claim that $\cH^\be_\infty(S)=0$, which combined with the inequality above implies contradiction and completes the proof. Indeed, if $z\in S$, then
$$
2+\al_0\le \Phi(0+,q(z+\cdot))\le \Phi(\infty, q(z+\cdot))=\Phi(\infty,q)=\la_*<5/2,
$$
in particular, $2+\al_0\le \Phi(0+,q(z+\cdot))<5/2$. By applying Theorem~\ref{thm:(n-1)-hom-Alm-freq}, we get $\dim_{\cH}(S)\le n-3$. This yields $\cH^\be_\infty(S)=0$, since $\be>n-3$.
\end{proof}

We now provide formal proof of Conjecture~\ref{conj:n-1} (c).
\begin{proof}
    [Proof of Conjecture~\ref{conj:n-1} (c)] Conjecture~\ref{conj:n-1} (c) follows from Lemmas~\ref{lem:anom-discrete-(n=3)} and ~\ref{lem:anom-(n-1)-dim}.
\end{proof}

\appendix


\section{Recursive Inequality}\label{appen:recur-ineq}
In this section we prove the following lemma concerning a recurrence inequality appearing in the proof of Theorem~\ref{thm:main} $(c)$.

\begin{lemma}\label{lem:M_k-est}
Given $k_0\in \mathbb{N}$ and positive constants $M$, $C_1$, $C_2$, let $\{M_k\}_{k=k_0}^\infty$ be a sequence of nonnegative numbers such that
$$
M_{k+1}\le M_k-C_1M_k^{2n-1}+C_2M_k^{2n-2}(1/k)^\be\,\,\,\text{ and }\,\,\, M_k\le M
\,\,\,\text{ for every } k\ge k_0, 
$$
where $\be=\frac1{2(n-1)}$. Then there exist constants $C_0>0$ and $K_0\in\mathbb{N}$, depending only on $n$, $k_0$, $M$, $C_1$, $C_2$, such that
\begin{align}\label{eq:M_k-induc}
M_k\le C_0(1/k)^\be,\quad k\ge K_0.
\end{align}
\end{lemma}

\begin{proof}
We take a constant $C_0=C_0(n,k_0,M,C_1,C_2)>0$ large so that
\begin{align}
    \label{eq:C_0-condi}
    \begin{split}
    &\max\{k_0,1\}\le 2C_2\left(\frac{C_0}2\right)^{2n-3},\quad    M\le \left(\frac{2^{2n-5}}{C_2}\right)^{\frac1{2n-2}}C_0^{\frac1{2n-2}},\\
    &2\le \frac{(2n-2)C_1C_0^{2n-2}}{2^{2n}},\quad 2^{\be(2n-1)}\le \frac{C_1C_0}{2^{2n}C_2}.
\end{split}\end{align}
We then let $K_0\in \mathbb{N}$ be the smallest integer such that $K_0\ge 2C_2\left(\frac{C_0}2\right)^{2n-3}$. Note that $K_0\ge k_0$ by \eqref{eq:C_0-condi}.

For such $C_0$ and $K_0$, we prove Lemma~\ref{lem:M_k-est} by induction on $k\ge K_0$. To this aim, we first prove \eqref{eq:M_k-induc} for $k=K_0$. As $2C_2\left(\frac{C_0}2\right)^{2n-3}>1$ by \eqref{eq:C_0-condi}, the construction of $K_0$ implies $K_0\le 4C_2\left(\frac{C_0}2\right)^{2n-3}$. This, along with \eqref{eq:C_0-condi}, gives
\begin{align*}
    C_0(1/K_0)^\be\ge C_0\left(\frac{2^{2n-5}}{C_2C_0^{2n-3}}\right)^{\frac1{2n-2}}=\left(\frac{2^{2n-5}}{C_2}\right)^{\frac1{2n-2}}C_0^{\frac1{2n-2}}\ge M\ge M_{K_0}.
\end{align*}
We now assume \eqref{eq:M_k-induc} is true for $k\ge K_0$ and prove it for $k+1$. We split its proof into the two cases 
$$
\text{either \,\,$M_k\le \frac{C_0}2\left(\frac1{k+1}\right)^{\be}$\,\, or\,\, $M_k> \frac{C_0}2\left(\frac1{k+1}\right)^{\be}$.}
$$
\emph{Case 1.} We first deal with the case $M_k\le \frac{C_0}2\left(\frac1{k+1}\right)^{\be}$. From
$$
M_{k+1}\le M_k+C_2M_k^{2n-2}(1/k)^\be\le \frac{C_0}2\left(\frac1{k+1}\right)^\be+C_2\left[\frac{C_0}2\left(\frac1{k+1}\right)^\be\right]^{2n-2}(1/k)^\be,
$$
we see that $M_{k+1}\le C_0\left(\frac1{k+1}\right)^{\be}$ follows once we show $C_2\left[\frac{C_0}2\left(\frac1{k+1}\right)^\be\right]^{2n-2}(1/k)^\be\le \frac{C_0}2\left(\frac1{k+1}\right)^\be$. A direct computation show that the latter inequality is equivalent to $k+1\ge C_2\left(\frac{C_0}2\right)^{2n-3}\left(\frac{k+1}k\right)^\be$. This obviously holds true for $k\ge K_0\ge 2C_2\left(\frac{C_0}2\right)^{2n-3}$.

\medskip\noindent\emph{Case 2.} Now we consider the case $M_k> \frac{C_0}2\left(\frac1{k+1}\right)^{\be}$. We have
\begin{align*}
    M_{k+1}&\le M_k-C_1M_{k}^{2n-1}+C_2M_k^{2n-2}(1/k)^\be\\
    &\le C_0(1/k)^\be-C_1\left[\frac{C_0}2\left(\frac1{k+1}\right)^\be\right]^{2n-1}+C_2\left[C_0\left(1/k\right)^\be\right]^{2n-2}(1/k)^\be.
\end{align*}
Thus, to obtain $M_{k+1}\le C_0\left(\frac1{k+1}\right)^\be$, it is enough to show that
$$
C_0(1/k)^\be+C_2\left[C_0\left(1/k\right)^\be\right]^{2n-2}(1/k)^\be\le C_1\left[\frac{C_0}2\left(\frac1{k+1}\right)^\be\right]^{2n-1}+C_0\left(\frac1{k+1}\right)^\be.
$$
This inequality follows once we prove
\begin{align*}
& (i)\,\,\,C_0(1/k)^\be\le \frac{C_1}2\left[\frac{C_0}2\left(\frac1{k+1}\right)^\be\right]^{2n-1}+C_0\left(\frac1{k+1}\right)^\be,\\
&(ii)\,\,\, C_2\left[C_0\left(1/k\right)^\be\right]^{2n-2}(1/k)^\be\le \frac{C_1}2\left[\frac{C_0}2\left(\frac1{k+1}\right)^\be\right]^{2n-1}.
\end{align*}
We observe that $(i)$ is equivalent to
$$
(1/k)^\be\le \frac{C_1C_0^{2n-2}}{2^{2n}}\left(\frac1{k+1}\right)^{\be(2n-1)}+\left(\frac1{k+1}\right)^\be.
$$
We multiply $(k+1)^\be$ on both sides and use $\be=\frac1{2n-2}$ to get 
$$
\left(\frac{k+1}{k}\right)^\be\le \frac{C_1C_0^{2n-2}}{2^{2n}}\left(\frac1{k+1}\right)+1,
$$
which is the same as 
$$
1+\frac1k\le \left(1+\frac{C_1C_0^{2n-2}}{2^{2n}}\left(\frac1{k+1}\right)\right)^{2n-2}.
$$
Multiplying out the right-hand side gives
$$
1+\frac1k\le 1+\frac{(2n-2)C_1C_0^{2n-2}}{2^{2n}}\left(\frac1{k+1}\right)+\cdots,
$$
which is true by \eqref{eq:C_0-condi} and the fact $\frac{k+1}k\le 2$.

It remains to consider $(ii)$. It is easy to see by a direction computation that $(ii)$ is equivalent to 
$$
\left(\frac{k+1}k\right)^{\be(2n-1)}\le \frac{C_1C_0}{2^{2n}C_2},
$$
which holds true again by \eqref{eq:C_0-condi} and the inequality $\frac{k+1}k\le 2$. This completes the proof.  
\end{proof}


\section*{Declarations}

\noindent {\bf  Data availability statement:} All data needed are contained in the manuscript.

\medskip
\noindent {\bf  Funding and/or Conflicts of interests/Competing interests:} The authors declare that there are no financial, competing or conflict of interests.


\end{document}